\documentclass[11pt]{amsart}
\usepackage{bm,graphics,color,amssymb,cite}
\usepackage[mathscr]{euscript}
\usepackage{floatrow}

\newcommand\cX{\mathscr X}

\newcommand\cY{\mathscr Y}

\newcommand\cN{\mathcal N}
\newcommand\cM{\mathcal M}

\newcommand\Ker{\mathop{\mathrm{Ker}}}
\newcommand\Span{\mathop{\mathrm{span}}}

\newcommand{\sign}{{\mathrm{sign}}}

\newcommand{\im}{\mathop{\mathrm{Im}}}
\newcommand{\re}{\mathop{\mathrm{Re}}}

\def\CC{\mathbb C}
\def\NN{\mathbb N}
\def\RR{\mathbb R}
\def\FF{\mathbb F}

\def\PP{\mathbb P}

\newtheorem{theorem}{Theorem}[section]

\newtheorem{lemma}[theorem]{Lemma}
\newtheorem{proposition}[theorem]{Proposition}

\theoremstyle{remark}
\newtheorem{example}[theorem]{Example}
\newtheorem{remark}[theorem]{Remark}
\newtheorem{question}[theorem]{Question}

\usepackage{enumerate}
\usepackage[dvipsnames]{xcolor}

\usepackage{float}
\usepackage{pdfpages}
\usepackage{multicol}

\begin{document}
\author[Lj. Aramba\v si\' c]{Ljiljana Aramba\v si\' c $^{1}$}\address{$^1$ Department of Mathematics, Faculty
of Science, University of Zagreb, Bijeni\v cka cesta 30, 10000 Zagreb, Croatia} \email{arambas@math.hr}

\author[A.  Guterman]{Alexander  Guterman $^{2}$}

\address{$^{2\ 5}$ Department of Mathematics and Mechanics, Lomonosov Moscow State University, Moscow, 119991, Russia}
\address{$^{2\ 5}$ Moscow Center for Continuous Mathematical Education, Moscow, 119002, Russia}
\address{$^{2\ 3\ 5}$ Moscow Center for Fundamental and Applied Mathematics, Moscow, 119991, Russia}
\address{$^{2}$ Moscow Institute of Physics and Technology, Dolgoprudny, 141701, Russia}
\email{guterman@list.ru} \email{s.a.zhilina@gmail.com}

\thanks{$^2$ $^5$ The work was financially supported by RSF grant 21-11-00283}

\author[B. Kuzma]{Bojan Kuzma $^3$}
\address{$^3$ University of Primorska, Glagolja\v{s}ka 8, SI-6000 Koper, Slovenia, and IMFM, Jadranska 19, SI-1000 Ljubljana, Slovenia}
\email{bojan.kuzma@upr.si}
\thanks{$^3$  This work is supported in part by the Slovenian Research Agency (research program P1-0285 and research project N1-0210).}

\author[R. Raji\' c]{Rajna Raji\' c $^4$}
\address{$^{4}$ Faculty of Mining, Geology and Petroleum Engineering, University of Zagreb, Pierottijeva 6, 10000 Zagreb, Croatia}
\email{rajna.rajic@rgn.hr}

\author[S.  Zhilina]{Svetlana  Zhilina $^5$}



\title[What does BJ orthogonality know\dots]{What does Birkhoff--James orthogonality know about the norm?}
\maketitle
\begin{abstract}
It is shown that Birkhoff--James orthogonality knows everything about the smooth norms in reflexive Banach spaces and can also  compute the dimensions of the underlying normed spaces.
\end{abstract}
\section{Introduction and preliminaries}
General  normed spaces are not equipped with an inner product.  Nonetheless, there do exist several nonequivalent
extensions of orthogonality from inner product spaces to general normed ones.  Arguably, one of the most  well-known in this respect is the Birkhoff--James orthogonality, introduced by Birkhoff~\cite{Birkhoff} and developed by
James~\cite{James2,James,James1}. It is defined in  an arbitrary normed space $(\cX,\|\cdot\|)$
over the field $\FF\in\{\RR,\CC\}$ as follows: the vector $x\in\cX$   \emph{is
Birkhoff--James orthogonal to $y\in\cX$} (shortly: BJ orthogonal; in symbols: $x\perp y$), if
$$\|x+\lambda y\|\ge \|x\|,\quad \forall \lambda\in \FF.$$

One easily sees that in inner product spaces BJ orthogonality is equivalent to the usual one. More importantly, several key properties of orthogonality in inner product spaces are inherited within BJ orthogonality.  For example,  among all the vectors from a subspace $\cY\subseteq\cX$ a vector $y_0\in\cY$ is the best approximation to $x$ (in the  sense that $\|x-y_0\|=\min_{y\in\cY}\|x-y\|$)
precisely when $x-y_0$ is BJ orthogonal to $\cY$. Another trivial implication, valid for every BJ orthogonal (i.e., $\cM\perp\cN$)
 closed
subspaces $\cM, \, \cN$  in a Banach space~$\cX$ is that they have trivial intersection (excuse the pun) and their sum, $\cM+\cN$ is also a closed subspace --- this   can be seen by noting that $\cM\perp\cN$ implies that the projection from $\cM+\cN$ onto $\cM$
is norm-one.
In this respect, we  mention here a classical result by Anderson~\cite{Anderson} that the range of an inner-derivation, induced by a normal operator $N$ on a complex Hilbert space, is always BJ orthogonal to its kernel (this result has been greatly extended in various directions \cite{Kit}).
Among other applications of BJ orthogonality we mention Bhatia--\v Semrl's alternative proof that the diameter of the unitary orbit of a given complex matrix $A$ equals $2d(A,\CC I)$ (see~\cite{BhatiaSemrl}).

It is an exercise  that BJ orthogonality is  in general nonsymmetric, that is, $x\perp  y$ does not always
imply  $y\perp  x$ (e.g., in the space  $(\FF^2,\|\cdot\|_{\infty})$ where $\|\cdot\|_\infty$ is the maximum norm we have
$(1,1)\perp  (1,0)$ but $(1,0)\not\perp  (1,1)$). In fact, James showed~\cite{James1} that in real spaces of dimension at least three BJ orthogonality is symmetric if and only if the norm is induced by an inner product.  His proof works also for complex normed spaces almost with no change, except that  instead of Kakutani's~\cite{Kakut} classification of real inner product spaces (being the ones where for each two-dimensional subspace there exists a norm-one  projection onto it), one has to use  Bohnenblust's~\cite{Bohnenblust} extension
 of Kakutani's classification to complex spaces. We refer the interested reader to the book by Amir \cite{Amir} where several other classifications of inner-product spaces are found,  including  a self-contained proof of Kakutani's classification (see \cite[\S12]{Amir}).

Thus, BJ orthogonality alone can detect if the norm comes from an inner product or not.  For two-dimensional spaces such a nice characterization does not exist though; the two-dimensional real vector spaces where BJ orthogonality is symmetric are called Radon planes, an example of which can be found in \cite[p.~561]{James1}. We refer to \cite{Martini-Swanepoel-Weiss} for much more information as well as historical remarks on this topic.

There are several additional characterizations of inner product spaces in \cite{James1} but they use  BJ orthogonality relation and additivity operation in unison.
One such, for example, is the following:  a real space $\cX$  of dimension at least three is an inner product space if and only if $x\perp  z$ and $y\perp  z$ always imply that $(x+y)\perp  z$ (see \cite[Theorem~2]{James1}). In a similar way, it was shown in~\cite{James} that BJ orthogonality together with additivity (and homogeneity operations)
can characterize smooth (and rotund, that is,  strictly convex) norms.

Our principal aim here is to study in more details what additional properties of the norm are possible to discern from BJ orthogonality alone, that is, without using it in unison with the underlying linear structure.  Similarly to James' and Bohnenblust's result that BJ orthogonality knows when the
space is Euclidean we show (in Section 2) that BJ orthogonality alone knows when the space is finite-dimensional; it can compute its dimension, it knows when the norm  is smooth and when it is strictly convex, and actually (see Section 3) knows everything about the smooth norms of reflexive Banach  spaces up to (conjugate) linear isometry.

 The easiest way to formulate and prove our results is to use the directed graph, denoted by $\hat{\Gamma}=\hat{\Gamma}(\cX)$, induced by the BJ orthogonality relation on a normed space $(\cX,\|\cdot\|)$. By definition, the vertex set of this (di)graph consists of one-dimensional subspaces of $\cX$
$$V(\hat{\Gamma})=\{[x]=\FF x;\;\;x\in\cX\setminus\{0\}\}$$
with two vertices  $[x],[y]\in V(\hat{\Gamma})$ forming a directed edge $[x]\to[y]$
if and only if $x\perp  y$. Observe that this relation is well-defined because BJ orthogonality is homogeneous in both factors. Observe
also that no nonzero vector is BJ orthogonal to itself, so $\hat{\Gamma}$ has no loops.
 Alternatively, we could define $\hat{\Gamma}$ over $\cX\setminus\{0\}$;  it seems, however, more natural to work
over projective spaces.

We call this graph  \emph{di-orthograph} to amplify the fact that its edges are directed. Now,  James' and Bohnenblust's result, which we already mentioned above, says that three or more dimensional normed spaces $\cX$ are inner product spaces if and only if the di-orthograph $\hat{\Gamma}=\hat{\Gamma}(\cX)$ is a simple graph, meaning that  for every directed edge $([x],[y])\in\hat{\Gamma}$, its opposite, $([y],[x])$, is also an edge in $\hat{\Gamma}$.

The  motivation behind our study is two-fold. Firstly, it has become clear recently that graphs of certain
relations defined on objects of a given category  carry a lot of information. For example, a commuting graph
whose vertices are all noncentral elements and two commuting elements are connected can distinguish a finite
nonabelian simple group among all groups~\cite{ABDOLLAHI-SHAHVERDI,Han-Chen-Guo,Sol-Wol} and can  also
distinguish among different Hilbert spaces~\cite{Kuzma}. Also, as it was  shown recently
in~\cite{Dol-Kuz-Sto1}, the graph related to Jordan  orthogonality distinguishes among finite-dimensional
formally real simple Jordan algebras.

As for the second motivation,  orthographs were used in the recent paper \cite{Dol-Kuz-Sto} where a
complete classification of maps which preserve orthogonality on finite-dimensional projective spaces over
reals, complexes, quaternions or octonions  is given without assuming surjectivity. Such a map is nothing but a
(possibly nonsurjective) homomorphism of the corresponding orthograph, and so it maps the maximum clique (i.e., the subset of pairwise connected vertices of maximal possible cardinality)  onto
another maximum clique. In this respect, the information about orthograph is useful for the investigations of
maps preserving the orthogonality  (cf. with~\cite{Blanco-Turnsek,Woj} where, in addition, linearity or at least additivity was assumed), which we are planning to do in a separate paper.

Before proceeding with our main results, let us remark that it is often easier to verify  BJ orthogonality with the help of James' criteria~\cite[Theorem 2.1]{James}  by which $x\perp y$ if and only if there exists a supporting functional $f_x$ at~$x$ (i.e., a  norm-one linear functional $f_x\colon \cX\to\FF$ with $f_x(x)=\|x\|$ which annihilates~$y$). One should mention here  that James in~\cite[Theorem 2.1]{James} assumed real normed spaces but the proof works also for the complex ones (the necessity  part of the proof only requires the fact that if unit vectors satisfy $x\perp y$, then a linear functional $f$ which maps $x$ to $1$ and annihilates~$y$  is norm-one on the span of $x,y$; using the Hahn--Banach  theorem one then extends the domain of this $f$ from $\Span_{\CC}\{x,y\}$ to $\cX$).  We refer also to  Kittaneh~\cite[Lemma 1]{Kit} for a shorter proof in real or complex spaces in case the norm is smooth at $x$.

By James' result, a line $[x]=\FF x\subseteq\cX $ is BJ orthogonal to a line $[y]=\FF y$ if and only if there exists a supporting functional $f_x$ at
some representative $x\in[x]$ which annihilates the line $[y]$. Recall that, by our definition,  a supporting
functional at a vector $x$  must satisfy $f_x(x)=\|x\|$ and $\|f_x\|=1$. We will call any such functional \emph{a supporting functional of a line $[x]$}. In particular, if $f_x$ is a supporting functional at a line $[x]$, then $\mu f_x$ is also, provided that $|\mu|=1$.

To simplify things we will frequently not
distinguish between a nonzero vector $x\in\cX$ and the line $[x]=\FF x\subseteq\cX$ passing through it. Thus, we will often denote
vertices of di-orthograph $\hat{\Gamma}(\cX)$ simply by $x$ instead of $[x]$ but we  do  call them lines.

Just before we finished our paper we learned that Tanaka very recently obtained similar results to ours in \cite{Tanaka1,Tanaka2};  in particular he also obtained a version of one of  our main results, i.e., Theorem~\ref{thm:iso} below. However,  our  methods are different and give some additional insights. Moreover, our Lemma~\ref{lem:sigma-preserves-orthogonality} answers  a question posted in \cite[section Remarks]{Tanaka2}.

\section{Main results -- property recognition}

The fundamental question which we address here is the following:
\begin{question}
Let ${\mathcal P}$ be a given property on a normed space. Can we decide, using BJ orthogonality alone, if a space has  property ${\mathcal P}$  or not?
\end{question} Of course, this  is too general for any hope of a definite answer and we will restrict ourselves to three very specific properties: the dimension, the smoothness, and the rotundness.

Let us first show that by means of $\hat{\Gamma}(\cX)$ one can calculate the dimension of a normed
space $\cX$.

\begin{proposition}\label{prop:inft-dim} A normed
space $\cX$ over $\FF$ is infinite-dimensional if and only if $\hat{\Gamma}(\cX)$ contains an infinite clique.
\end{proposition}

\begin{proof}
See our recent paper \cite[Corollary 3.1 and Theorem 3.5]{Ara-Gut-Kuz-Raj-Zhi}.
\end{proof}

\begin{lemma}\label{lem:dim}
Let $\cX$ be a normed space over $\FF$ and $n\ge 2$. Then the following statements are equivalent.
\begin{itemize}
	\item[(i)] For any $k\in\{1,\dots,n-1\}$ and any $k$-tuple of lines $(x_1,\dots,x_k)$ satisfying $x_i \perp x_j$, $1 \le i < j \le k$, one can always find $x_{k+1}\in\hat{\Gamma}(\cX)$ with $x_i \perp x_{k+1}$, $1 \le i \le k$.
	\item[(ii)] For any $(n-1)$-tuple of lines $(x_1,\dots,x_{n-1})$, one can always find $x_n\in\hat{\Gamma}(\cX)$ with $x_i \perp x_n$, $1 \le i \le n-1$.
	\item[(iii)] $\dim\cX\ge n$.
\end{itemize}
\end{lemma}
\begin{proof}
	We will once again rely on the fact that $x\perp  y$ if and only if  $y\in\Ker f_x$, where $f_x$ is a
supporting functional at $x$.

	 (iii) $\Longrightarrow$ (ii). Choose an arbitrary $(n-1)$-tuple of lines $(x_1,\dots,x_{n-1})$. Let $f_i$ be any supporting functional at $x_i$, $i=1,\dots,n-1$. Denote $\cX_{n-1}=\bigcap_{i=1}^{n-1}\Ker f_i$.
    Then $\mathrm{codim}\,\cX_{n-1}\le n-1$ together with $\dim\cX\ge n$ imply that the subspace $\cX_{n-1}\subseteq \cX$  contains at least one line, $x_n$. Clearly, $x_i \perp x_n$ for $i = 1,\dots,n-1$.

	(ii) $\Longrightarrow$ (i) is straightforward.
	
	$\neg$ (iii) $\Longrightarrow$ $\neg$ (i). Assume $\dim\cX=k\le n-1$.  From \cite[Theorem 25.5, p.~246]{Rockafellar} and the fact that the norm is positively homogeneous and $x\mapsto \lambda x$ is a differentiable map  it follows that the norm's unit ball has a
	smooth point~$\hat{x}_1$. Let $f_1$ be a unique supporting functional at $\hat{x}_1$ and let
$x_1:=[\hat{x}_1]$
	be the corresponding line in $\hat{\Gamma}(\cX)$.  Then $x_1\perp  y$ for a line $y$ is equivalent to $y\subseteq\Ker f_1=\cX_1$.
		Now, the induced norm in  the $(k-1)$-dimensional space $\cX_1$ again has a smooth point, say $\hat{x}_2$;
	let $f_2\colon\cX_1\to\FF$ be the corresponding supporting functional which by the Hahn--Banach theorem we
extend to a supporting
	functional at a line $x_2:=[\hat{x}_2]$ on the original space $\cX_0:=\cX$.  Recursively,  one obtains  a $k$-tuple of lines
	$(x_1,\dots,x_{k})$, where $x_i\subseteq\cX_{i-1}\subseteq\cX_{i-2},$ $2\le
i\le k,$  is
	a line spanned by a smooth point $\hat{x}_i \in \cX_{i-1}:= \cX_{i-2}\cap \Ker f_{i-1}$ and has a unique supporting functional $f_i\colon
	\cX_{i-1}\to\FF$ which we again extend by the Hahn--Banach theorem to a supporting functional defined on~$\cX$. Clearly, $i<j$
	implies that  $x_j\subseteq\Ker f_i$, so $x_i\perp  x_j$ for $1\le i<j\le k$. 	
	Assume a line $x\subseteq\cX$ satisfies $x_i\perp  x$ for each $1\le i\le k$.  Since $x_1\perp  x$ and
	$\hat{x}_1\in\cX$ is smooth, then $x\subseteq\Ker f_1=\cX_1$.  Since $\hat{x}_2\in\cX_1$ is smooth and $x_2\perp  x$, we
	further have $x\subseteq\Ker f_2\cap \cX_1=\cX_2$. Proceeding recursively, one finds
	$x\subseteq\cX_k=0$, a contradiction.
\end{proof}

\begin{remark}\label{rem:dim}
Note that a normed
space $\cX$ has dimension $n$ if and only if the di-orthograph $\hat{\Gamma}(\cX)$ satisfies item
(i) with $k=1,\dots,n-1$
and does not satisfy item (i) for larger $k$, that is, there exist $n$-tuples of lines
$x_1,\dots,x_n\in\hat{\Gamma}(\cX)$ with $x_i\perp  x_j$ for $i<j$ but no further line
$x\in\hat{\Gamma}(\cX)$ satisfies $x_i\perp  x$ for $1\le i\le n$. Hence, the dimension of $\cX$ can be
calculated from the di-orthograph $\hat{\Gamma}(\cX)$ alone.
\end{remark}

Next we show that $\hat{\Gamma}(\cX)$ can detect the presence of nonsmooth points.
It is known that the norm on $\cX$ is smooth if and only if BJ orthogonality is right
additive in the sense that $x\perp  y_1$ and $x\perp  y_2$ always implies $x\perp (y_1+y_2)$ (see \cite[Theorem 4.2]{James}).
Below we provide another criterion for smoothness which does not rely on underlying additive structure.

\begin{lemma}\label{lem:smooth-norm}
	Let $\cX$ be a Banach space over $\FF$ with  $2\le\dim\cX=n<\infty$.  Then the norm in $\cX$ is nonsmooth if and only if
	there  exist $n-1$ lines $x_1,\dots,x_{n-1}\in\hat{\Gamma}(\cX)$ and two additional distinct lines
$y_n,y_n'$
	such that $x_i\perp  x_j$ for $1\le i<j\le n-1$,  and $x_i\perp  y_n$,  $x_i\perp  y_n'$ for $1\le i\le
	n-1$.
\end{lemma}

\begin{proof}
	If $\cX$ has a smooth norm, then each  of the lines $x_i=[\hat{x}_i]$  has a unique supporting
		functional $f_i$. From
	$x_i\perp  x_j$, $i<j$, it follows, by induction, that $x_i\subseteq\bigcap_{k=1}^{i-1} \Ker f_k$. This implies
	that the functionals $f_k$ are linearly independent, for otherwise pick the smallest integer $i$ with
	$f_i=\sum_{k=1}^{i-1}\alpha_k f_k$ and evaluate at $\hat{x}_i$ to deduce $1=f_i(\hat{x}_i)=\sum_{k=1}^{i-1}\alpha_k
	f_k(\hat{x}_i)=\sum_{k=1}^{i-1}\alpha_k\cdot 0=0$, a contradiction. In particular, $\dim\bigcap_{k=1}^{n-1} \Ker f_k=1$, and there is only one line $y_n=\bigcap_{k=1}^{n-1} \Ker f_k$ which satisfies $x_i\perp y_n$  $(i=1,\dots,n-1)$.

Assume now that the norm has a nonsmooth point, say at a vector $x_{n-1},$ and let $f_{n-1},f_{n-1}'$ be
two
	linearly independent supporting functionals at $x_{n-1}$. Choose a vector $y_n$ such that $f_{n-1}(y_n)=0$
and
	$f_{n-1}'(y_n)\neq 0$ and choose a nonzero $y_n'\in\Span\{x_{n-1},y_n\}$ with $f_{n-1}'(y_n')= 0$. In
	particular,
	$$x_{n-1}\perp  y_n\quad\hbox{and}\quad x_{n-1}\perp  y_n'.$$ 	
	Recursively, choose
	norm-one functionals $f_{k}$ and unit vectors $x_{k}$ for $k=n-2,\dots,1$ such that: 	
	(i) $y_n,x_{n-1},\dots,x_{k+1}\in \Ker f_k,$ 	
	(ii) the norm of $f_k$ is attained at $x_k$. 	
	It follows that
	$f_k$ is a supporting functional at $x_k$ and that $x_i\perp  x_j$ and $x_i\perp  y_n$ for $1\le i<j\le
n-1$.
	Also, $y_n'\in\Span\{x_{n-1},y_n\}\subseteq \Ker f_i$ for $i=1,\ldots,n-2$  implies $x_i\perp  y_n'$ for
$1\le
i\le n-2$.
By a slight abuse of language we next identify a nonzero vector with a line spanned by it. Thus, the lines
$x_1,\dots, x_{n-1},y_n,y_n'\in\hat{\Gamma}(\cX)$ satisfy the claim.
\end{proof}

We proceed to classify rotund (i.e., strictly convex) norms. Given a vertex~$z$ in di-orthograph
$\hat{\Gamma}$, let $N_z:=\{v\in\hat{\Gamma};\;\;(z,v)\in E(\hat{\Gamma})\}$  denote its neighborhood; here $E(\hat{\Gamma})$ is the set of all directed edges of $\hat{\Gamma}$. Note that  $N_z\neq\varnothing$ if the underlying normed space is at least two-dimensional.

\begin{lemma} \label{lemma:not-strictly-convex}
 A normed space $\cX$  with $\dim\cX\ge 2$ over $\FF$ is strictly convex if and only if the function
$\hat{\Gamma}(\cX)\to 2^{\hat{\Gamma}(\cX)}$ which maps a vertex $z$ to its neighborhood $N_z$ is injective.
\end{lemma}
\begin{proof}
Suppose $\cX$ is not strictly convex and let ${\bf L}:=\{\lambda x+(1-\lambda)y;\;\;0\le \lambda\le 1\}$ be a nondegenerate line on the norm's unit
sphere~$S$. Consider distinct unit vectors $z_1,z_2$ from the relative interior of ${\bf L}\subseteq S$.
Then for some $\varepsilon>0$ we have
\begin{equation*}
\|(1-\lambda) z_1+\lambda z_2\| =1;\qquad \lambda\in(-\varepsilon,1+\varepsilon).
\end{equation*}
Let now $f$ be an arbitrary supporting functional at $z_1$, that is, ${f(z_1) = \| f \| = 1}$.
Then
$$|1 + \lambda (f(z_2)-1)|=|f((1-\lambda) z_1+\lambda z_2)| \leq \|
f \| \cdot \| z_1 + \lambda (z_2 - z_1) \| = 1$$ for any $\lambda
\in (-\varepsilon, \varepsilon)$, so $f(z_2)=1$.
Hence $f$ is a supporting functional for both $z_1$ and $z_2$, and thus the sets of supporting functionals at $z_1$
and at $z_2$ coincide.  This implies that the set $\{x\in\cX;\;\;z_1\perp  x\}$ coincides with the set
$\{w\in\cX;\;\;z_2\perp  w\}$ because, by~\cite{James}, $z \perp  x$ if and only if $x$ is annihilated by some
supporting functional at $z$. As such, also $N_{z_1}=N_{z_2}$.

To prove the converse we assume $\cX$ is strictly convex. We need to show that for any two distinct vertices
$z_1, z_2 \in \hat{\Gamma}(\cX)$ we have $N_{z_1} \neq N_{z_2}$.

Note first that $x \perp  y$ in $\cX$ if and only if $x \perp  y$ in $\Span \{ x, y \} \subseteq \cX$, since,
by the definition, BJ orthogonality of two vectors depends only on the norms of their linear combinations. We
now pass to the subspace $\cY = \Span \{z_1, z_2\}$. Clearly, the norm on $\cY$ is also strictly convex. It is
sufficient to show that $N_{z_1} \neq N_{z_2}$ in~$\hat{\Gamma}(\cY)$.

Consider some supporting functional $f$ at $z_1$ in $\cY$, and let $x \in \Ker f$. Hence $z_1 \perp  x$, so $x
\in N_{z_1}$. Assume, to reach a contradiction, that $x \in N_{z_2}$. Then there exists a supporting functional
$g$ at $z_2$ in $\cY$ such that $x \in \Ker g$.
Since $\dim \cY = 2$ and $\Ker f \cap \Ker g$ is nontrivial, it follows that $\Ker f = \Ker g$ which implies that $g$ is a unimodular scalar multiple of $f$.
By transferring this  scalar multiple to $z_2$ we may assume that $f(z_1) =
f(z_2) = \|f\| = \|z_1\| = \|z_2\| = 1$, so it follows that
$$
1 = f(\lambda z_1 + (1 - \lambda) z_2) \le \|f\| \cdot \|\lambda z_1 + (1 - \lambda) z_2\| \le 1
$$
for all $\lambda \in [0,1]$. Thus the line segment $[z_1,z_2]$ consists of unit vectors, a contradiction to
the strict convexity. Hence $x \notin N_{z_2}$ and $N_{z_1} \neq N_{z_2}$.
\end{proof}

\section{Main results -- Isomorphism problem}

How much information on the norm is encoded in the di-orthograph $\hat{\Gamma}(\cX)$? First we note that, if
$\cX$ and $\cY$ are isometrically isomorphic Banach spaces, then $\hat{\Gamma}(\cX)$ and $\hat{\Gamma}(\cY)$
are isomorphic. More precisely, if $A\colon\cX\to \cY$ is a linear bijective isometry then a line $x$ is BJ
orthogonal to a line $y$, relative to the norm on $\cX$, if and only if $Ax$ is BJ orthogonal to $Ay$ relative to the norm on
$\cY.$ That is, $x\mapsto Ax$ induces an isomorphism of di-orthographs $\hat{\Gamma}(\cX)$ and $\hat{\Gamma}(\cY)$. Our main result (see Theorem~\ref{thm:iso} below) will establish  the converse of this fact.

Let us start with some preliminary definitions and results.  Recall that, according to~\cite[part 3, ch. 1, \S 2]{Beauzamy}, in a real smooth normed space ${(\RR^n,\|\cdot\|)}$  the unique
supporting functional $f_x$ at a unit vector $x=(x_1,\dots,x_n)\in\RR^n$ is given by the norm's Gateaux directional derivative at $x$. In fact, it equals
\begin{equation}\label{eq:f_x=grad}
f_x=\langle \cdot, \partial\|x\|\rangle
\end{equation}
 where $\langle x,y\rangle:=xy^t$ is the standard scalar product of row vectors in $\RR^n$ and $\partial\|x\|:=\left(\frac{\partial\|\cdot\|}{\partial x_1},\dots,\frac{\partial\|\cdot\|}{\partial x_n}\right)(x)$ is the   norm's differential evaluated at vector $x$.

 A complex normed space $(\CC^n,\|\cdot\|)$  can be regarded as a real normed space $(\RR^{2n}, \| \cdot \|)$ by restricting the scalars.  We write $x=(x_1,\dots,x_n)\in\CC^n$ and consider the real and complex components of $x_k=\re x_k+i\,\im x_k$ as independent real variables. The $\RR$-linear functional $f_x$ defined above which depends on $2n$ real variables $(\mathop{\re }x_1, \mathop{\im } x_1, \dots, \mathop{\re }x_n, \mathop{\im } x_n)$ is the real part of a $\CC$-linear functional given by $\tilde{f}_x\colon z\mapsto f_x(z)-if_x(iz)$. Its  norm coincides with $\|f_x\|=1$, and therefore it must satisfy   $\tilde{f}_x(x)=1$, so it is a $\CC$-linear supporting functional at a unit vector $x$. It follows easily that $\tilde{f}_x$ is given by
\begin{equation*}
 \tilde{f}_x\colon z\mapsto   \left\langle z,\left( \frac{\partial{\|x\|}}{\partial\re  x_1}+i\frac{\partial{\|x\|}}{\partial\im
  x_1},\dots, \frac{\partial{\|x\|}}{\partial\mathop{\re }x_n}+i\frac{\partial{\|x\|}}{\partial\im
  x_n}\right)\right\rangle,
\end{equation*}
where $\langle x,y\rangle:=xy^\ast$ is the standard scalar product of row vectors in $\CC^n$. This can be simplified if one introduces complex partial  derivatives
$$\frac{\partial\|x\|}{\partial x_k}:=  \frac{\partial{\|x\|}}{\partial\mathop{\re }
  x_k} - i\frac{\partial{\|x\|}}{\partial\im
  x_k}\quad \hbox{ and }\quad \frac{\partial\|x\|}{\partial \overline{x_k}}:=  \frac{\partial{\|x\|}}{\partial\mathop{\re }
  x_k} + i\frac{\partial{\|x\|}}{\partial\im
  x_k};$$
then we can identify the $\CC$-linear supporting functional  $\tilde{f}_x$  with the norm's complex
conjugate  differential
\begin{align*}
\partial_{\CC}\|x\|:&=\left(\frac{\partial \|x\|}{\partial
\overline{x_1}},\dots,\frac{\partial \|x\|}{\partial
\overline{x_n}}\right)\\
&=\left( \frac{\partial{\|x\|}}{\partial\re  x_1}+i\frac{\partial{\|x\|}}{\partial\im
  x_1},\dots, \frac{\partial{\|x\|}}{\partial\re  x_n}+i\frac{\partial{\|x\|}}{\partial\im
  x_n}\right).
\end{align*}

The words \emph{a curve} and \emph{a path} may have several different meanings, depending on the context and the literature. Throughout the present paper, \emph{a curve} denotes a  subset in $\CC$, which is the image of \emph{a path}, i.e., the image of a continuous map ${\bf r}\colon [a,b]\to\CC$ where $[a,b]\subseteq\RR$ is an interval with at least two different points.  We say that a function $f\colon\CC\to\CC$ is bounded from above on a subset $\Gamma\subseteq\CC$ if $\sup_{z\in\Gamma} |f(z)|<\infty$.

\begin{lemma}\label{lem:sigma-bounded-on-curve}
Suppose a nonzero ring homomorphism $\sigma\colon\CC\to\CC$ is bounded from above on a curve $\Gamma\subseteq\CC$ with more than one point. Then $\sigma$ is continuous, hence either identity or a complex conjugation.
\end{lemma}
\begin{proof}
We will  rely  freely on the  fact that if an additive and multiplicative map $\sigma\colon \CC\to\CC$ is bounded from above on a set $\Gamma$, then it remains bounded from above on a  set $\Gamma_1$  obtained from $\Gamma$ by  applying any of the   following four procedures finitely many times:
(i) taking a  subset, i.e., $\Gamma_1\subseteq\Gamma$, (ii) translating by $\mu\in\CC$, i.e.,  $\Gamma_1=\mu+\Gamma$,  (iii) rotating/expanding, i.e., $\Gamma_1=\mu \Gamma$, and (iv) taking the union of finitely many  sets obtained in previous iterations.

Let ${\bf r}\colon[0,1]\to \Gamma$ be the corresponding non-constant path. By restricting and parameterizing  the domain  we can assume that ${\bf r}(0)\neq {\bf r}(1)$. By translating and rotating/expanding the resulting curve we can further assume that ${\bf r}(0)=0$ and ${\bf r}(1)=1$.

 If the curve $\Gamma$ contains a line segment then we can further translate and rotate/expand it to assume that this line segment equals the interval $[0,1]$. Then $\sigma$ is bounded on the closed unit disk because if $a,b\in[0,1]$, then $\sigma(a+ib) =\sigma(a) + \sigma(i)\sigma(b)$. As such, $\sigma$ is continuous  (see, e.g.,  \cite[Corollary~5, p.~15]{Acz-Dho}) and the result follows.

  Assume $\Gamma$ contains no line segment. Being compact, it is contained in the smallest solid rectangle $Q$ whose sides are parallel to real/imaginary axis. Now, the union $\Gamma\cup (1+\Gamma)$ is still a curve because the end point of a path ${\bf r}$ in $\Gamma$ coincides with the initial point of a translated path $t\mapsto 1+{\bf r}(t)$ of $(1+\Gamma)$. By induction we see that the finite union
  $\bigcup_{k=0}^{n-1}(k+\Gamma)$ which, by the abuse of notation we again denote by $\Gamma$, is still a curve which connects $0$ to $n$; here an integer $n$ is chosen to be  bigger than four-times the height of a rectangle~$Q$. By another abuse of notation we denote  the union of translated solid rectangles $\bigcup_{k=0}^{n-1} (k+Q)$ again by $Q$.
  We now rotate $\Gamma$ by~$\pi/2$, i.e., we consider $i\Gamma$, and translate $\Gamma$ as well as its rotated version. The resulting union (with negative sign in the exponent denoting that the parametrization of $\Gamma$ must be reversed),
  $$\hat{\Gamma}:=\Gamma^-\cup (i\Gamma)\cup (ni+\Gamma)\cup (n+i\Gamma^-)$$
  is a closed curve, which connects in clock-wise direction $0$ to $ni$, to $n(i+1)$ to $n$ and back to $0$. Also, this closed curve is contained in the union of four solid rectangles $$\hat{Q}:=Q\cup (i Q)\cup (ni+Q)\cup(n+iQ)$$
  each obtained by rotating $Q$ by $\pi/2$ and/or translating it. Since $n$ is at least  four times the height of $Q$, we see that the above union properly encircles, but does not contain a point $\alpha:=\frac{n}{2}(1+i)$.

\begin{figure}[H]
  \floatbox[{\capbeside\thisfloatsetup{capbesideposition={right,bottom},capbesidewidth=7cm}}]{figure}[\FBwidth]
{\caption{\label{figure:Q}The original curve $\Gamma$ (with dashed lines) and its translates/rotations (with dotted lines) form a closed curve $\hat{\Gamma}$ sitting inside the union $\hat{Q}$ of four solid rectangles  (with border lines drawn). }}
{\includegraphics[width=4cm]{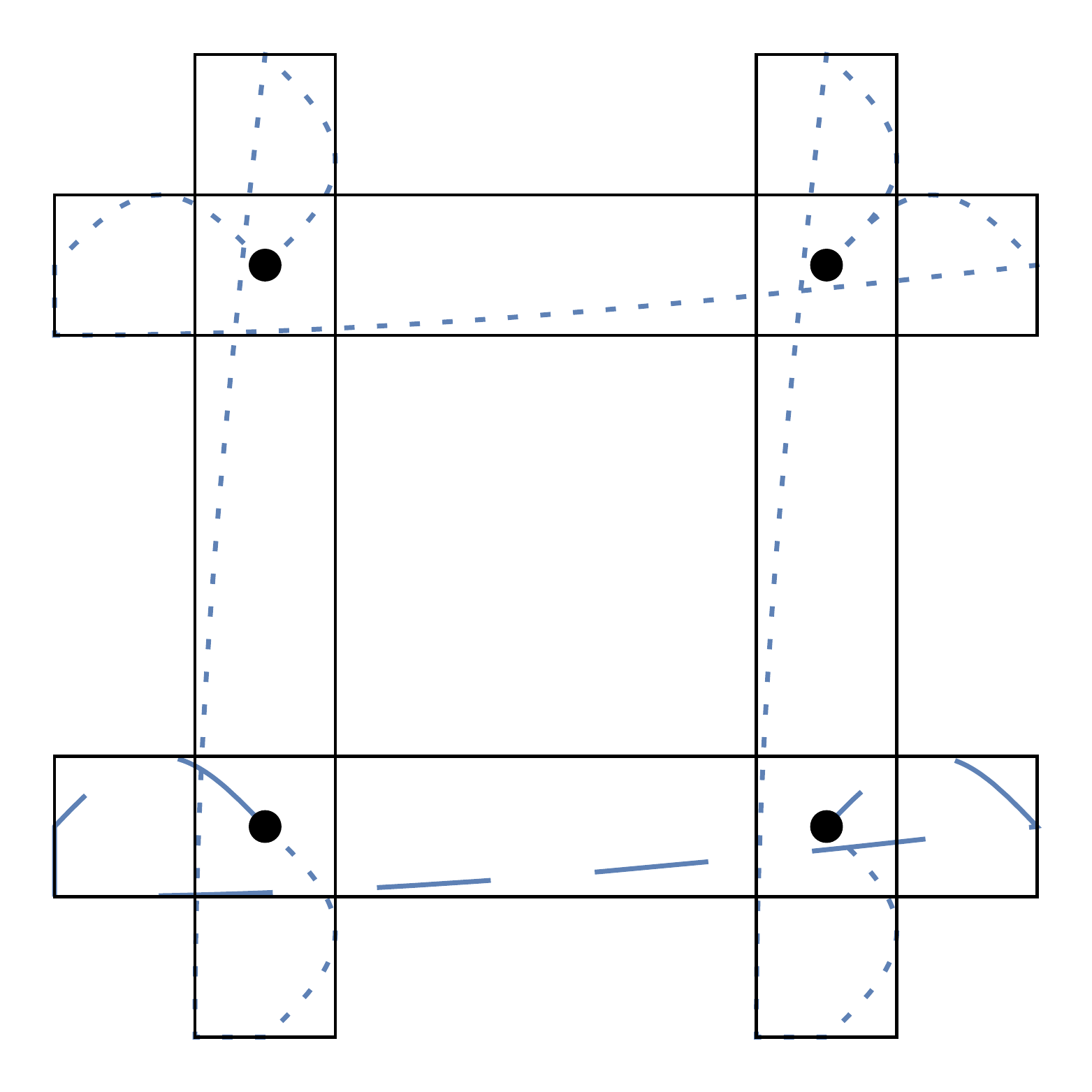}}
\end{figure}

  Notice that  a clock-wise oriented rectangle $\hat{R}:=[0,n]\cup i[0,n]\cup(ni+[0,n])\cup(n+i[0,n])\subseteq\hat{Q}$, i.e., the closed polygonal curve  which connects  $0 \rightarrow
  (in)\rightarrow n(1+i)\rightarrow n\rightarrow 0$ is a deformation of $\hat{Q}$ within $\hat{Q}$ (that is, the identity map on $\hat{Q}$ is homotopic, within $\hat{Q}$,  to a map from $\hat{Q}$ onto $\hat{R}$). This deformation then also induces a deformation of $\hat{\Gamma}$ onto $\hat{R}$ within $\hat{Q}$ which never passes through $\alpha:=\frac{n}{2}(1+i)$ nor through $\infty$. Since $\hat{R}$ separates $\CC$ by the Jordan Curve Theorem, the same holds for $\hat{\Gamma}$ by \cite[XVII, 4.3]{Dugundji}, and in fact $\alpha$ is in a bounded component of $\CC\setminus\hat{\Gamma}$.

Since $\sigma$ is additive, it is clearly  bounded from above also on the subset $\hat{\Gamma}-\hat{\Gamma}\subseteq\CC$. We claim that this subset has a nonempty interior, so it must also contain a line segment, and hence, as shown above, $\sigma$ is continuous.   To see that $\hat{\Gamma}-\hat{\Gamma}\subseteq\CC$ contains interior points, we borrow the idea from a paper by Simon and Taylor \cite{Simon-Taylor}.

Let $C_\infty$ be the unbounded component of $\CC\setminus\hat{\Gamma}$ and let $C_0:=\CC\setminus (C_\infty\cup\hat{\Gamma})$ be the union of its bounded components. Notice that $C_\infty$ is an open set whose boundary is contained in $\hat{\Gamma}$ (see, e.g., \cite[XVII, 1.2]{Dugundji}),
so that $(C_\infty\cup\hat{\Gamma})=\overline{C_\infty}\cup\hat{\Gamma}$ is closed and as such its complement, $C_0$ is open.

Let $$\Omega:=\{x\in\CC;\;\;\exists \,a_1,a_2\in\hat{\Gamma}: a_1\in x+C_0\,\hbox{ and }\,a_2\in x+C_\infty\}.$$
This is clearly an open subset. It is also nonempty: in fact, let $\alpha\in C_0$ be a point in the bounded component. Consider the two horizontal rays emanating from $\alpha$, that is,   $L^{-}:=(-\infty,0]+\alpha $ and $L^{+}:=\alpha+[0,\infty) $. Since $\hat{\Gamma}$ separates $\CC$, both rays must intersect it. Let $a_1\in\hat{\Gamma}\cap L^{-}$ be the closest point to $\alpha$ on  the left ray which also belongs to a compact closed curve $\hat{\Gamma}$ and let $a_2\in\hat{\Gamma}\cap L^+$ be the furthest point from $\alpha$ on the right ray   which still belongs to $\hat{\Gamma}$. By moving $\hat{\Gamma}$ slightly to the left, i.e., by adding small enough real $x<0$ to $\hat{\Gamma}$   we achieve that $a_1\in\hat{\Gamma}$ belongs to the bounded component of $\CC\setminus(x  +\hat{\Gamma})$, that is, $a_1\in x+C_0$ while $a_2\in \hat{\Gamma}$ belongs to its unbounded component, that is, $a_2\in x+C_\infty$. Indeed, $\Omega $ is nonempty.

Assume there is $x\in\Omega$ such that $x\notin\hat{\Gamma}-\hat{\Gamma}$. Then $(x+\hat{\Gamma})\cap\hat{\Gamma} = \varnothing$. It follows that a connected set $-x+\hat{\Gamma}$ is completely contained in the union of disjoint open subsets $C_\infty\cup C_0$ (namely, $-x+\hat{\Gamma}\subseteq\CC = C_0\sqcup \hat{\Gamma}\sqcup C_\infty$, the disjoint union), so $\hat{\Gamma}\subseteq(x+C_0)\sqcup (x+C_\infty)$. Also, by definition of $\Omega$, both open sets $(x+C_0)$ and $(x+C_\infty)$ contain a portion of $\hat{\Gamma}$, so they partition a connected set $\hat{\Gamma}$, a contradiction. Indeed, $\Omega$ is an open subset contained in $\hat{\Gamma}-\hat{\Gamma}$.
\end{proof}
\begin{remark}
The  proof of the lemma relies on the fact that if a closed curve $\Gamma\subseteq\CC$ separates the complex plane, then $\Gamma-\Gamma$ has nonempty interior.
With the help of dimension theory for separable metric spaces (see a book by Hurewitz and  Wallman~\cite{Hure-Wall}) much more can be said. In fact, Shchepin~\cite{Scepin} showed the following. If $\Gamma_1,\dots,\Gamma_n$ are $n\ge 1$ compact connected subsets in $\RR^n$ such that (i) $0\in\bigcap_{i=1}^n \Gamma_i$ and  (ii) there exist $n$ linearly independent points $a_i\in\Gamma_i$, then the sum $\sum\Gamma_i:=\{x_1+\dots+x_n;\;\; x_i\in\Gamma_i\}\subseteq\RR^n$ is of dimension at least $n$ and therefore it contains an open ball (see \cite[Theorem IV.3]{Hure-Wall}).

Note that Shchepin's result no longer holds if we drop the condition that $\Gamma_1,\cdots,\Gamma_n$ contain linearly independent points.  Say, if $\Gamma_1=\Gamma_2=[0,1]$ (this can be made into  a closed curve by going from $0$ to $1$ and then back), then $\Gamma_1+\Gamma_2=[0,2]\subseteq\CC$ is of dimension $1$.
\end{remark}

Our  next lemma is known in the case of real spaces, see Taylor~\cite[Theorem~2]{Taylor}. Since we will use its variant for complex, two-dimensional spaces we present the short proof,
valid for real and complex normed spaces, for the sake of convenience. We remark that Taylor used a similar
approach (see \cite[Theorem~1]{Taylor}).

\begin{lemma} \label{lemma:minimal-clique}
Let $n\ge2$ and let ${\mathscr X}$ be an $n$-dimensional normed space over~${\mathbb F}$. Then there exist
at least~$n$ unit vectors which are pairwise BJ orthogonal.
\end{lemma}

\begin{proof} We may assume that ${\mathscr X}={\mathbb F}^n$.
Let ${\bf S}\subseteq{\mathbb F}^n$ denote the norm's unit sphere. The  continuous real-valued function
\begin{equation*}
f\colon(x_1,\dots,x_n)\mapsto |\det(x_1, \dots, x_n)|
\end{equation*}
attains its maximum on a compact set $\prod_1^n {\bf S}$, say at a point with coordinates
 $y_1, \dots, y_n \in {\bf S}$.  Clearly, $f(y_1,\dots,y_n)>0$ because ${\bf S}$ contains $n$ linearly independent vectors.  We now show that $y_1, \dots, y_n$ are pairwise
mutually BJ orthogonal.

Assume the contrary. Let $k \neq j$ be such that $y_k \not\perp y_j$, that is, $$\| y_k + \lambda y_j \| <
\| y_k \|=1$$ for some $\lambda \in{\mathbb F}$. By their linear independence  we have $ y_k + \lambda y_j
\neq 0$, so
 $$0< \| y_k + \lambda y_j \|<1.$$ Consider now $y_k' = \frac{y_k + \lambda
y_j}{\| y_k + \lambda y_j \|} \in {\bf S}$. Then
\begin{multline*}
\!\! f(y_1,\dots,y_{k-1},y_k',y_{k+1},\dots,y_n)
   = \frac{|\det(y_1, \dots,y_{k-1}, y_k + \lambda y_j,y_{k+1}, \dots, y_n)|}{\| y_k + \lambda y_j \|} \\
    = \frac{|\det(y_1, \dots, y_k, \dots, y_n)|}{{\| y_k + \lambda y_j \|}}  = \frac{1}{{\| y_k + \lambda y_j \|}} f(y_1,\dots,y_n) > f(y_1,\dots,y_n),
\end{multline*}
a contradiction.
\end{proof}
It was shown by R\"atz \cite{Ratz} (see also  Sundaresan~\cite[Lemma~1]{Sundaresan}) that  BJ orthogonality in real normed spaces is Thalesian, that is, if $x,y$ are BJ orthogonal vectors in a real normed space, then for every $\lambda_0>0$ there exists a scalar $\alpha$ such that $(x+\alpha y)\perp  (\lambda_0 x-\alpha y)$. This fact was  used by W\'ojcik~\cite[Proof of Theorem 3.1]{Woj} to give an alternative proof that linear maps which preserve BJ orthogonality between real normed spaces are scalar multiples of isometries.  The main idea of the proof of the next lemma comes from
W\'ojcik's paper (see \cite[Proof of Theorem 3.1]{Woj}) and  uses a partial extension of Thalesian property for BJ orthogonality in \emph{complex} normed spaces:
we show that if unit vectors $x,y$ are mutually BJ orthogonal, then there exists a curve $\Gamma\subseteq\CC$ with more than one point, such that for every $\lambda\in\Gamma$ we can find $\alpha\in\CC$ with $(x+\alpha y) \perp (\lambda x-\alpha y)$.

Recall that an additive map $\Phi\colon\cX\to\cY$ between $\FF$-vector spaces $\cX$ and~$\cY$  is $\sigma$-quasilinear if  $\Phi(\lambda x)=\sigma(\lambda)\Phi(x)$ holds
where $\sigma\colon\FF\to\FF$ is a field homomorphism (if $\sigma$ is surjective, such maps are semilinear).

\begin{lemma}\label{lem:sigma-preserves-orthogonality}
    Let $\cX,\cY$ be   smooth normed spaces of dimension at least two over $\FF$, and let $\sigma\colon\FF\to\FF$ be a field homomorphism.
 If  a nonzero $\sigma$-quasilinear map $\Phi\colon\cX\to\cY$  preserves BJ orthogonality, then $\sigma$ is identity or a complex conjugation.
\end{lemma}
\begin{proof}
For real spaces this is clear because identity is the only nonzero field homomorphism of $\RR$. In the sequel we assume $\cX,\cY$ are complex normed spaces, so $\FF=\CC$.

By Lemma~\ref{lemma:minimal-clique}
there exist two mutually BJ orthogonal unit vectors $x,y\in\cX$ such that also $\Phi(x)\neq 0$.  We identify the two-dimensional normed subspace $\Span_{\CC}\{x,y\}\subseteq\cX$ with the space  $(\CC^2,\|\cdot\|)$ of row vectors  in such a way that  $x,y$  are identified with standard basis vectors, that is, $x=(1,0)$ and $y=(0,1)$.  Recall that a $\CC$-linear supporting functional at a point $(1,\alpha )\in\CC^2$ is given~by
$$f_{\alpha}=\left(\frac{\partial\|(1,\alpha)\|}{\partial \overline{z_1}}\,,\,\frac{\partial\|(1,\alpha)\|}{\partial \overline{z_2}}\right)^\ast, \quad (z_k=x_k+i y_k).$$
We have
$$ \overline{ \left( \frac{\partial\|(1,\alpha)\|}{\partial \overline{z_k}} \right)}  = \frac{\partial \overline{\|(1,\alpha)\|}}{\partial z_k} = \frac{\partial\|(1,\alpha)\|}{\partial z_k}
$$
for $k = 1, 2$, so the kernel of $f_{\alpha}$ contains a row vector
\begin{equation}\label{eq:w_alpha}
\left( \frac{\partial\|(1,\alpha)\|}{\partial z_2} \,,\, - {\frac{\partial\|(1,\alpha)\|}{\partial z_1}}\right).
\end{equation}
Therefore, for every scalar $\alpha$ we get $(1,\alpha )\perp\left(\frac{\partial\|(1,\alpha)\|}{\partial z_2} \,,\, - {\frac{\partial\|(1,\alpha)\|}{\partial z_1}} \right)$.

Since $x,y$ are mutually BJ orthogonal, their supporting functionals equal $$\left(\frac{\partial\|(1,0)\|}{\partial \overline{z_1}},\frac{\partial\|(1,0)\|}{\partial \overline{z_2}}\right)^\ast=\mu_x\cdot(1,0)^\ast \; \hbox{ and } \left(\frac{\partial\|(0,1)\|}{\partial \overline{z_1}},\frac{\partial\|(0,1)\|}{\partial \overline{z_2}}\right)^\ast=\mu_y\cdot (0,1)^\ast,$$ for some nonzero $\mu_x,\mu_y\in\CC$, respectively.  Since partial derivatives of a smooth norm are continuous (see Rockafellar \cite[Theorem 25.5]{Rockafellar}), the first supporting functional equals the limit of $\left(\frac{\partial\|(1,\alpha)\|}{\partial \overline{z_1}},\frac{\partial\|(1,\alpha)\|}{\partial \overline{z_2}}\right)^\ast
$ as $\alpha\in\RR$ goes to~$0$. Next, positive homogeneity of the norm implies $\frac{\partial\|(1,\alpha)\|}{\partial \overline{z_k}}=\frac{\partial\|(1/\alpha,1)\|}{\partial \overline{z_k}}$ ($\alpha>0$), so  the  second supporting functional equals the limit of $\left(\frac{\partial\|(1,\alpha)\|}{\partial \overline{z_1}},\frac{\partial\|(1,\alpha)\|}{\partial \overline{z_2}}\right)^\ast$ as $\alpha\in\RR$ goes to~$\infty$.
Hence, the vector
$$
w_{\alpha}=\frac{\frac{\partial\|(1,\alpha)\|}{\partial z_2}}{\frac{\partial\|(1,\alpha)\|}{\partial z_1}}\alpha x - \alpha y
$$
is a well-defined function of $\alpha$ around $\alpha=0$, and is clearly parallel to \eqref{eq:w_alpha}, so $(x+\alpha y)\perp w_\alpha$.
Also, its first component, i.e.,
$$\lambda(\alpha):=\frac{\frac{\partial\|(1,\alpha)\|}{\partial z_2}}{\frac{\partial\|(1,\alpha)\|}{\partial z_1}}\alpha,$$
is a continuous function which cannot be constantly zero because the numerator and  the denominator are nonconstant functions and cannot vanish simultaneously.
As such, with $\alpha$ restricted to a suitable closed interval $I=[0,\varepsilon]$, its range $\Gamma$ is a curve with more than one element in $\CC$.

Thus, for every  $\alpha\in [0,\varepsilon]$ we have $(x+\alpha y)\perp(\lambda(\alpha) x-\alpha y)$. Since we assumed that $\sigma$-quasilinear $\Phi$ preserves BJ orthogonality, we get that
$$\Big(\Phi(x)+\sigma(\alpha)\Phi(y)\Big)\perp \Big(\sigma(\lambda(\alpha))\Phi(x)-\sigma(\alpha)\Phi(y)\Big).$$
From here, and as $x\perp y$ imply $\Phi(x)\perp\Phi(y)$, the very definition of BJ orthogonality of $\Phi(x)$ and $\Phi(y)$ gives
\begin{equation*}
\begin{aligned}
\|\Phi(x)\|&\le \|\Phi(x)+\sigma(\alpha)\Phi(y)\|=\|\Phi(x+\alpha y)\|\\
&\le\|\Phi(x+\alpha y)+\Phi(\lambda(\alpha) x-\alpha y)\|=|1+\sigma(\lambda(\alpha))|\cdot\|\Phi(x)\|
\end{aligned}
\end{equation*}
so that $|\sigma(\mu)|$ is bounded from below by $1$ on the set  $\mu=1+\lambda(\alpha) \in 1+\Gamma\subseteq\CC$. Being multiplicative, it follows that $\sigma$ is bounded from above on the reciprocal set $\frac{1}{1+\Gamma}$, which is also a curve with more than one element. The result then follows from Lemma~\ref{lem:sigma-bounded-on-curve}.
 \end{proof}

The main result of this chapter was essentially proven  by Blanco and Turn\v sek in \cite[Corollary 3.4]{Blanco-Turnsek} (but see also  recent papers by Tanaka~\cite{Tanaka1,Tanaka2}). The proof given there assumed an infinite-dimensional reflexive smooth space and also assumed the domain coincides with the range. Only minor adaptations are needed to classify isomorphisms instead of automorphisms between infinite-dimensional spaces. Additional lemmas are required to show continuity in  finite-dimensional spaces. We decided to present their arguments,   with all the necessary modifications, for the sake of convenience.
\begin{theorem}\label{thm:iso}
    Let $(\cX,\|\cdot\|_{\cX})$ and $(\cY,\|\cdot\|_{\cY})$ be smooth 
    Banach spaces over $\FF$ with $\cX$ reflexive and  $3\le\dim \cX\le\infty$. If $T\colon \hat{\Gamma}(\cX) \to \hat{\Gamma}(\cY)$  is a di-orthograph isomorphism, then $\dim\cY=\dim\cX$ and
 there exists a linear or conjugate-linear surjective map $U\colon \cX \to\cY$  such that
 $$\|Ux\|_{\cY}=\|x\|_{\cX}.$$
\end{theorem}
\begin{proof}
    By Proposition~\ref{prop:inft-dim},
 $\dim\cX<\infty$ if and only if $\dim \cY<\infty$.  It then follows from Lemma~\ref{lem:dim} and Remark~\ref{rem:dim} that $\dim\cX=\dim\cY$.

Define a map $S\colon\PP\cX^\ast\to\PP\cY^\ast$ between projectivizations of  dual spaces of $\cX$ and $\cY$ by $[f]\mapsto [f_{T[x]}]$. Here, $x$ is  (any) unit vector where  $f$ achieves its norm (it exists because of reflexivity of $\cX$)
and $f_{T[x]}$ is a unique supporting functional at
some nonzero vector in a line $T[x]\in\PP\cY$. Note  that $S$ is well-defined because if $f$ achieves its norm on unit vectors $x,y\in\cX$, then $f$ is a unique supporting functional for $[x]$ and $[y]$, giving that $N_{[x]}=N_{[y]}=\Ker f$ (notations copied from Lemma~\ref{lemma:not-strictly-convex}).
Therefore, also $N_{T[x]}=N_{T[y]}$ and since $\cY$ is smooth, the supporting functionals at $T[x],T[y]\in\PP\cY$ coincide.

The last argument shows at once that $S$ is not only  well defined but also  injective.  Also, the surjectivity of $T$ implies  that the range of $S$ consists of all lines spanned by norm-attaining functionals.
It is easily seen that, given $f\in\cX^*$ and $z \in\cX$, we have $f(z) = 0$ if and only if,
for every functional $g \in  S[f]$ and vector $y \in T[z]$, we have $g(y) = 0$.

We claim that $T$ is a morphism of projective spaces. Namely, assume otherwise that there exist the lines $[x],[y],[z]\in\hat{\Gamma}$ which satisfy
\begin{equation}\label{eq:morphism}
      [z]\subseteq[x]+[y]
\end{equation} but their images $T[x],T[y],T[z]$ span a three-dimensional space. Then we can choose a norm-attaining unit functional $g\colon\cY\to\FF$ which annihilates $T[x]$ and $T[y]$ but not $T[z]$.
Clearly, the line spanned by such $g$ belongs to the range of $S$, so $[g]=S[f]$, and $f$ annihilates $[x]$ and $[y]$ but not the line $[z]$ contradicting \eqref{eq:morphism}.

 Indeed, $T$ is a bijective morphism of projective spaces. By the (nonsurjective version of) Fundamental Theorem of Projective Geometry (see Faure~\cite{Faure}),
 there
exist
a field isomorphism $\sigma\colon\FF\to\FF$ and a $\sigma$-quasilinear map  $V\colon \cX \to\cY$  such that $T[x] = [V x]$ for $x\in\cX \setminus \{0\}$. Clearly, $V$ must be orthogonality preserving, so, by  Lemma
\ref{lem:sigma-preserves-orthogonality}, $\sigma$ is either identity or a complex conjugation.
It follows by the main result of \cite{Blanco-Turnsek} (but see also \cite{Woj}) that the (conjugate) linear orthogonality preserving map $V\colon (\cX,\|\cdot\|_{\cX})\to(\cY,\|\cdot\|_{\cY})$ is a scalar multiple of an isometry (i.e., there exists a scalar $\mu\in\CC$ such that $U:=\mu V$ satisfies $\|Ux\|_{\cY}=\|x\|_{\cX}$).
\end{proof}
\begin{remark}
If $\cY$ is reflexive, then in the proof above the map $S$ is surjective and one could use   Moln\'ar's result, \cite[Corollary 1]{Molnar}, to show that $T$ is induced by a $\sigma$-quasilinear bijection.
\end{remark}

\begin{remark} If $\cX$ is finite-dimensional and smooth, then in the theorem above the smoothness of  $\cY$ need not be assumed in advance;  it follows  by Lemma~\ref{lem:smooth-norm} (after Lemma~\ref{lem:dim} and Remark~\ref{rem:dim} establish that  $\dim\cX=\dim\cY$).
\end{remark}

\section{Concluding examples and remarks}

In Banach spaces of dimension two the graphs related to two norms can be isomorphic, even if
the two norms are not isomorphic (i.e., Theorem~\ref{thm:iso} may fail).
\begin{example}
Consider an odd integer $p\ge 3$ and let $\|\cdot\|_p$
be the  $p$-norm on the two-dimensional real sequence
space~$\RR^2$. This is clearly a smooth norm with  the differential
$$
\partial \|(x,y)\|_p= \frac{(|x|^{p-1}\sign\,x,|y|^{p-1}\sign\,y)}{(|x|^p+|y|^p)^{1-1/p}},  \quad (x,y)\neq (0,0).
$$
For every $(x_0,y_0)\in \RR^2,$ the functional $(x,y) \mapsto \langle (x,y),\partial\|(x_0,y_0)\|_p\rangle$
has a kernel spanned by $(|y_0|^{p-1}\sign\,y_0, -|x_0|^{p-1}\sign\,x_0).$ By
\eqref{eq:f_x=grad} it follows that a unique line which is BJ orthogonal to a line $\RR(x,y)$ equals
$$\RR(|y|^{p-1}\sign\,y, -|x|^{p-1}\sign\,x).$$
Since $p$ is odd, the function $t\mapsto |t|^{p-1}\sign\,t$ is bijective on $\RR$.  Hence, for a given line
${\bf x}=\RR(x,y)$ there exists a unique line ${\bf a}=\RR(a,b)$ with ${\bf a}\perp {\bf x}$, it is given by
a unique solution of the two equations $|a|^{p-1}\sign\,a=-y$ and $|b|^{p-1}\sign\,b=x$. Thus, the di-orthograph $\hat{\Gamma}(\RR^2,\|\cdot\|_p)$ consists of copies of a  bilateral infinite directed path connecting (not
necessarily disjoint) vertices
$$\cdots\rightarrow v_{-1}\rightarrow v_0\rightarrow v_1\rightarrow\cdots.$$
If $v_0 = \RR(x,y)$ then for $k \in \NN$ we have $v_k=\RR f_p^k(x,y)$, where $f_p^k=f_p\circ\dots\circ f_p$ is a $k$-fold compose of
$$f_p\colon(x,y)\mapsto (|y|^{p-1}\sign\,y, -|x|^{p-1}\sign\,x).$$
One computes that
$$f_p^k(x,y)=\begin{cases}
(|x|^{(p-1)^k} \sign\,x, \ |y|^{(p-1)^k} \sign\,y), & k\equiv 0\pmod 4;\cr
(|y|^{(p-1)^k} \sign\,y,\ -|x|^{(p-1)^k} \sign\,x), & k\equiv 1\pmod 4;\cr
(-|x|^{(p-1)^k} \sign\,x,\ -|y|^{(p-1)^k} \sign\,y), & k\equiv 2\pmod 4;\cr
(-|y|^{(p-1)^k} \sign\,y,\ |x|^{(p-1)^k} \sign\,x), & k\equiv 3\pmod 4.\cr
\end{cases}$$
 Assume that $v_0=v_k$ for some $k \in \NN$ in this path.
Notice that $ \RR(x,y)=\RR f_p^k(x,y) $  if and only if the determinant of a matrix with rows $(x,y)$ and
$f_p^k(x,y)$ vanishes. By using the fact that $p-1$ is even, so $|x|^{(p-1)^k} = x^{(p-1)^k}$, this reduces,
for $k$ even into $(x,y)\in\RR(1,0)$ or $(x,y)\in\RR(0,1)$ or $(x,y)\in\RR(1,1)$ or $(x,y)\in\RR(1,-1)$. For
$k$ odd it reduces into $|x|^s+|y|^s=0$ for  a suitable integer $s$ in which case there are no solutions.

Hence, the di-orthograph $\hat{\Gamma}(\RR^2,\|\cdot\|_p)$ consists of copies of bilateral infinite directed paths
with pairwise disjoint vertices and exactly two disjoint undirected edges $\RR(1,0)\leftrightarrow\RR(0,1)$
and $\RR(1,1)\leftrightarrow\RR(1,-1)$ (if the graph contains an edge together with its opposite edge we call
such edge undirected).

 Thus, for all odd $p$ such graphs are isomorphic but clearly the two-dimensional $\|\cdot\|_p$-spaces are
nonisometric for different $p\in(1,\infty)$.
\end{example}

There is a close relation between a Banach space $\cX$ and its dual $\cX^\ast$. This reflects also in di-orthographs:

\begin{example}
Let $(\cX,\|\cdot\|)$ be a smooth and strictly convex reflexive Banach space and let
$(\cX^\ast,\|\cdot\|^\ast)$ be its dual. The procedure which to each line $x\subseteq\cX$  associates a unique
line  containing the supporting functional $f_x$ at  $x$ induces an antiisomorphism between  di-orthographs
$\hat{\Gamma}(\cX)$ and $\hat{\Gamma}(\cX^\ast)$.

This follows from the observation that $x\perp  y$ if and only if $f_x(y)=0$ where $f_x$ is the (unique)
supporting functional at $x$. Notice that the point evaluation functional $F_x\colon\cX^\ast \to \FF$, given by
$f\mapsto f(x)$, is a supporting functional at $f_x$. Thus, $x\perp  y$ implies that $F_y(f_x)=f_x(y)=0$ giving
$f_y\perp  f_x$. The converse implication follows by reflexivity and the fact the dual norm $\|\cdot\|^\ast$
is also smooth because $\cX$ is strictly convex.
\end{example}

\begin{remark}
Clearly, in $\dim\cX\ge 3$ there are infinitely many lines in the kernel of a nonzero linear functional and
hence, given $x\in\hat{\Gamma}(\cX)$, there are infinitely many vertices $y\in\hat{\Gamma}(\cX)$ with
$x\rightarrow y$.  The converse is also true: given $x\in\hat{\Gamma}(\cX)$, there exist
infinitely many vertices $y$ with $y\rightarrow x$. Indeed, take an arbitrary $z \in \hat{\Gamma}(\cX)$ which is not parallel to $x$, and let $\cY = \Span\{x,z\}$. Consider a unit functional $f: \cY \to \FF$ which annihilates $x$ and choose a unit vector $y \in \cY$ where $f$ achieves its norm. Then $y \perp x$ in $\cY$ and, therefore, in $\cX$, so we have $y \rightarrow x$. Clearly, $y \not\in \Span\{x\}$. Since $\dim \cX \ge 3$ and $z$ is arbitrary, we obtain an infinite number of $y$ such that $y \rightarrow x$.

In the case when $\FF = \RR$, the proof can be obtained geometrically. One can imagine a hyper-cylindrical surface circumscribed around the norm's unit ball in $\cX$, whose axis is parallel to $x$. This surface has an infinite number of points of contact with the norm's unit ball, and each point of contact $y$ satisfies $y \rightarrow x$ in $\hat{\Gamma}(\cX)$.
\end{remark}

Our final example accompanies Lemma~\ref{lemma:not-strictly-convex}.
\begin{example} If  $N_x=N_y$ for some unit vectors, then $x,y$ share the same supporting functional.

    Namely, suppose otherwise and consider the two-dimensional subspace $\cN:=\Span_{\FF}\{x,y\}$. By the Hahn--Banach theorem, $x,y$ also do not share the same supporting functional inside a two-dimensional subspace $\cN$. Note that there does exist a supporting functional $f_x\colon\cN\to\FF$ at $x$, and hence there exists some unit $z\in\cN$  in its kernel. Clearly, then $x\perp z$, so $z\in N_x=N_y$. Therefore, $y\perp z$, and hence $z$ is annihilated by some supporting functional $f_y\colon \cN\to\FF$ at $y$.
    Since
    $\cN$ is two-dimensional and $\Ker f_x \cap \Ker f_y$ is nontrivial, it follows that $\Ker f_x = \Ker f_y$ which implies that $f_y$ is a unimodular scalar multiple of $f_x$, a contradiction.
   \end{example}
\begin{remark}
 Note that the converse implication does not hold. Consider, e.g., $x=(1,1)$ and $y=(1,0)$ in $(\RR^2,\|\cdot\|_\infty)$, which share the same supporting functional but $N_x\neq N_y$.
\end{remark}


\begin{thebibliography}{99}


 \bibitem{ABDOLLAHI-SHAHVERDI} A. Abdollahi and H. Shahverdi, Characterization of the alternating group by its     non-commuting graph. J. Algebra {\bf 357} (2012), 203--207.
 \bibitem{Acz-Dho} J. Aczel and J. Dhombres, Functional equations in several variables. Series: Encyclopedia of Mathematics and its Applications, vol. 31, Cambridge UP 1989.
 \bibitem{Amir} D. Amir, Characterizations of inner product spaces.
Operator Theory: Advances and Applications, 20. Birkh\"auser Verlag, Basel, 1986.
 \bibitem{Anderson} J. Anderson, On normal derivations.     Proc. Amer. Math. Soc. {\bf 38} (1973), 135--140.
  \bibitem{Ara-Gut-Kuz-Raj-Zhi} Lj. Aramba\v si\' c,  A.  Guterman, R. Raji\' c,  B. Kuzma,  and S. Zhilina, Symmetrized Birkhoff--James orthogonality in arbitrary normed spaces.  J. Math. Anal. Appl. {\bf 502~(1)} (2021), 125203.
 \bibitem{Beauzamy} B. Beauzamy. Introduction to Banach spaces and their geometry. North-Holland Mathematics Studies, vol. 68 (Amsterdam: North-Holland, 1982).
 \bibitem{BhatiaSemrl} R. Bhatia and P. \v Semrl, Orthogonality of matrices and some distance problems.    Linear     Algebra Appl. \textbf{287} (1999), 77--86.
 \bibitem{Birkhoff} G. Birkhoff,  Orthogonality in linear metric spaces.  Duke Math. J. {\bf 1} (1935),     169--172.
 \bibitem{Blanco-Turnsek} A.  Blanco and A. Turn\v{s}ek, On maps that preserve orthogonality in normed spaces.     Proc.    Roy.    Soc. Edinburgh Sect. A {\bf 136 (4)} (2006), 709--716.
 \bibitem{Bohnenblust} F. Bohnenblust, A characterization of complex  Hilbert spaces.  Portugal Math. {\bf 3} (1942), 103--109.

 \bibitem{Dol-Kuz-Sto}  G. Dolinar, B. Kuzma, and N.  Stopar,  Characterization of orthomaps on the Cayley plane.     Aequationes Math. {\bf 92 (2)} (2018), 243--265.
 \bibitem{Dol-Kuz-Sto1}  G. Dolinar, B. Kuzma, and N.  Stopar,  The orthogonality relation classifies formally real     simple Jordan algebras. Commun. Algebra {\bf 48 (6)} (2020), 2274--2292.
 \bibitem{Dugundji} J. Dugundji, Topology. Wm.~C.~Brown Publishers,  Dubuque, Iowa, 1989.
 \bibitem{Faure} C.-A. Faure,  An elementary proof of the fundamental theorem of projective geometry. Geom. Dedicata. \textbf{90} (2002), 145--151.
 \bibitem{Han-Chen-Guo} Z. Han, G. Chen, and X. Guo, A characterization theorem for sporadic simple groups. Sib.     Math. J. {\bf 49} (2008), 1138--1146.

   \bibitem{Hure-Wall} W. Hurewics and H. Wallman, Dimension theory.  Princeton University Press, 1941.
 \bibitem{James} R.C. James, Orthogonality and linear functionals in normed linear spaces. Trans. Amer. Math.     Soc. {\bf 61} (1947), 265--292.
 \bibitem{James1} R.C. James, Inner products in normed linear spaces. Bull. Amer. Math. Soc. \textbf{53}     (1947),   559--566.
 \bibitem{James2} R.C. James, Orthogonality in normed linear spaces. Duke Math. J. \textbf{12}     (1945), 291--302.
 \bibitem{Kakut} S. Kakutani, Some characterizations  of Euclidean spaces.  Jpn. J. Math. {\bf 16} (1939), 93--97.

 \bibitem{Kit} F. Kittaneh. Operators that are orthogonal to the range of a derivation. J. Math. Anal. Appl. {\bf 203} (1996), 868--873.

 \bibitem{Kuzma} B. Kuzma,  Dimensions of complex Hilbert spaces are determined by the commutativity relation.     J. Operator Theory {\bf 79 (1)} (2018), 201--211.
 \bibitem{Martini-Swanepoel-Weiss}  H. Martini,   K.J. Swanepoel, and G. Wei\ss,  The geometry of Minkowski spaces—a survey. I,     Expo. Math. {\bf 19 (2)} (2001), 97--142.
 \bibitem{Molnar} L. Moln\'ar, Orthogonality preserving transformations on indefinite inner product spaces:     generalization of Uhlhorn's version of Wigner's theorem. J. Funct. Anal. {\bf 194 (2)} (2002), 248--262.
 \bibitem{Ratz}J. R\"atz, On orthogonally additive mapping. Aequationes Math. {\bf 28} (1985) 35--49.
 \bibitem{Rockafellar}  R.T. Rockafellar, Convex Analysis. Princeton U.P., 1970.
\bibitem{Scepin} E.V. Shchepin, The dimension of a sum of curves. Uspekhi Mat. Nauk   {\bf 30  4(184)} (1975), 267--268 (in Russian).
\bibitem{Simon-Taylor} K. Simon and K. Taylor, Interior of sums of planar sets and curves.     Math. Proc. Cambridge Philos. Soc. {\bf 168 (1)} (2020),  119--148.
 \bibitem{Sol-Wol} R. Solomon and A. Woldar, Simple groups are characterized by their non-commuting graphs. J.     Group Theory {\bf 16} (2013), 793--824.
  \bibitem{Sundaresan} K. Sundaresan, Orthogonality and nonlinear functionals on Banach spaces. Proc. Amer. Math. Soc. {\bf 34} (1972), 187--190.
  \bibitem{Tanaka1}  R. Tanaka,  Nonlinear equivalence of Banach spaces based on Birkhoff-James orthogonality. J. Math. Anal. Appl. {
\bf 505 (1)} (2022), Paper No. 125444, 12 pp.
\bibitem{Tanaka2} R. Tanaka,  On Birkhoff-James orthogonality preservers between real non-isometric Banach spaces. arXiv:2108.00655.


 \bibitem{Taylor}  A.E. Taylor, A geometric theorem and its application to biorthogonal systems. Bull. Amer. Math. Soc. {\bf 53} (1947), 614--616.
 \bibitem{Woj}  P. W\'{o}jcik,  Mappings preserving B-orthogonality. Indag. Math. (N.S.) {\bf 30 (1)} (2019),  197--200.

\end{thebibliography}
\end{document}